\documentclass[graybox]{svmult}

\usepackage{type1cm}        % activate if the above 3 fonts are
\usepackage{makeidx}         % allows index generation
\usepackage{graphicx}        % standard LaTeX graphics tool
\usepackage{multicol}        % used for the two-column index
\usepackage[bottom]{footmisc}% places footnotes at page bottom

\usepackage{newtxtext}       % 
\usepackage{newtxmath}       % selects Times Roman as basic font

\usepackage[super]{cite}

\usepackage{amsfonts}
\usepackage{amsmath}
\usepackage{amssymb}
\usepackage{amscd}
\usepackage{mathrsfs}
\usepackage{bbm}

\newcommand{\be}{\begin{equation}}
\newcommand{\ee}{\end{equation}}

\newcommand{\bes}{\begin{equation*}}
\newcommand{\ees}{\end{equation*}}

\newcommand{\bH}{\mathbb{H}}

\newcommand{\N}{\mathbb{N}}

\newcommand{\R}{\mathbb{R}}

\newcommand{\mbS}{\mathbb{S}}

\newcommand{\A}{\mathbb{A}}

\newcommand{\mcL}{\mathcal{L}}

\newcommand{\mcT}{\mathcal{T}}

\newcommand{\bx}{\boldsymbol{x}}
\newcommand{\cF}{\mathcal{F}}

\newcommand{\sH}{\mathscr{H}}

\newcommand{\Span}{\mathrm{span\,}}

\DeclareMathOperator\Lip{Lip}

\DeclareMathOperator\Sc{Sc}
\DeclareMathOperator\Ve{Vec}

\newcommand{\vertiii}[1]{{\left\vert\kern-0.25ex\left\vert\kern-0.25ex\left\vert #1 
    \right\vert\kern-0.25ex\right\vert\kern-0.25ex\right\vert}}

\makeindex             % used for the subject index
\begin{document}

\title*{Hypercomplex Iterated Function Systems}
\author{Peter Massopust}
\institute{Peter Massopust \at Technical University of Munich, Center of Mathematics, Boltzmannstrasse 3, 85748 Garching by Munich, Germany, \email{massopust@ma.tum.de}
%\and Name of Second Author \at Name, Address of Institute \%email{name@email.address}
}
\maketitle

\abstract*{We introduce the novel concept of hypercomplex iterated function system (IFS) on the complete metric space $(\mathbb{A}_{n+1}^k,d)$ and define its hypercomplex attractor. Systems of hypercomplex function systems arising from hypercomplex IFSs and their backward trajectories are also introduced and it is shown that the attractors of such backward trajectories possess different local (fractal) shapes.}

\abstract{We introduce the novel concept of hypercomplex iterated function system (IFS) on the complete metrizable space $(\A_{n+1}^k,d)$ and define its hypercomplex attractor. Systems of hypercomplex function systems arising from hypercomplex IFSs and their backward trajectories are also introduced and it is shown that the attractors of such backward trajectories possess different local (fractal) shapes.}
\section{Preliminaries and Notation}
This paper intends to merge two areas of mathematics: Clifford Algebra and Analysis and Fractal Geometry. The former has a long successful history of extending concepts from classical analysis and function theory to a noncommutative division algebra setting and the latter has developed into an area of bourgeoning research over the last few decades. Attractors of so-called iterated function systems (IFSs) are fractal sets in the sense given first by B. Mandelbrot \cite{mandel}. Fractals generated by IFSs have the following significant approximation property, called the Collage Theorem\cite{barnsley}: Every nonempty compact subset of a complete metric space can be approximated arbitrarily close (in the sense of the Hausdorff metric) by the attractor of an IFS, that is, by a fractal set. (For the precise mathematical statement, we refer to \cite{barnsley}.) This property has important and fundamental implications in image compression and image analysis; cf. for instance, \cite{BH}. Moreover, it has initiated the construction of fractal functions and fractal surfaces and their application to approximation and interpolation theory.  A collection of related results can be found in, e.g.,\cite{massopust}. This monograph also provides the fundamental result by D. Hardin\cite{hardin}, namely, that every continuous compactly supported and refinable function, i.e., every wavelet, is a piecewise fractal function.

Here, we initiate the inclusion of fractal techniques into the Clifford setting. The non-commutativity generates more intricate fractal patterns. The structure of this paper is as follows. After a short and terse introduction to Clifford algebras and IFSs, we define hypercomplex IFSs and their attractors. The final section extends these concepts to systems of function systems as defined in \cite{LDV}.
\subsection{A Brief Introduction to Clifford Algebras}
In this section, we give a terse introduction to the concept of Clifford algebra, mainly, to set notation and terminology, and refer the reader to, for instance, \cite{Clifford,Clifford2,Hyper} . To this end, denote by $\{e_1, \ldots, e_n\}$ the canonical basis of the Euclidean vector space $\R^n$. The real Clifford algebra $C\ell(n)$ generated by $\R^n$ is defined by the multiplication rules $e_i e_j + e_j e_i = -2 \delta_{ij}$, $i,j\in \{1,\ldots, n\} =: \N_n$, where $\delta_{ij}$ is the Kronecker symbol. The dimension of $C\ell(n)$ regarded as a real vector space is $2^n$. 

An element $x\in C\ell(n)$ can be represented in the form $x = \sum\limits_{A} x_A e_A$ with $x_A\in \R$ and $\{e_A : A\subseteq \N_n\}$, where $e_A := e_{i_1} e_{i_2} \cdots e_{i_m}$, $1\leq i_1 < \cdots < i_m \leq n$, and $e_\emptyset =: e_0 := 1$. A conjugation on Clifford numbers is defined by $\overline{x} := \sum\limits_{A} x_A \overline{e}_A$ where $\overline{e}_A := \overline{e}_{i_m} \cdots \overline{e}_{i_1}$ with $\overline{e}_i := -e_i$ for $i\in\N_n$, and $\overline{e}_0 := e_0 = 1$. The Clifford norm of the Clifford number $x = \sum\limits_{A} x_A e_A$ is $|x| := \sqrt{\sum\limits_{A} |x_A|^2}$. 

An important subspace of $C\ell(n)$ is the space of hypercomplex numbers or paravectors. These are Clifford numbers of the form $x = x_0 + \sum\limits_{i=1}^n x_i e_i$. The subspace of hypercomplex numbers is denoted by $\A_{n+1} := \Span_\R\{e_0, e_1, \ldots, e_n\} = \R\oplus \R^n$. Given a Clifford number $x\in C\ell(n)$, we assign to $x$ its hypercomplex or paravector part by means of the mapping $\pi: C\ell(n)\to \A_{n+1}$, $x \mapsto x_0 + \sum\limits_{i=1}^n x_i e_i$. 

Note that each hypercomplex number $x$ can be identified with an element $(x_0, x_1, \ldots, x_n) =: (x_0, \bx)\in \R\times \R^n$. For many applications in Clifford theory, one therefore identifies $\A_{n+1}$ with $\R^{n+1}$.

The scalar part, $\Sc$, and vector part, $\Ve$, of a hypercomplex number $\A_{n+1}\ni x = x_0 + \sum\limits_{i=1}^n x_i e_i$ is given by $x_0$ and $\boldsymbol{x} = \sum\limits_{i=1}^n x_i e_i$, respectively. 

The {conjugate} $\overline{x}$ of the hypercomplex number $x = s + \bx$ is the hypercomplex number $\overline{x}= s - \bx$. The Clifford norm of $x\in \A_{n+1}$ is given by $|x| = \sqrt{x \overline{x}} = \sqrt{s^2 + |\bx|^2} = \sqrt{s^2+\sum\limits_{i=1}^n x_i^2}$.

We denote by $M_k (\A_{n+1})$ the right module of $k\times k$-matrices over $\A_{n+1}$. Every element $H = (H_{ij})$ of $M_k (\A_{n+1})$ induces a right linear transformation $L: \A_{n+1}^k\to C\ell(n)^k$ via $L(x) = H x$  defined by $L(x)_i = \sum\limits_{j=1}^k H_{ij} x_j$, $H_{ij}\in \A_{n+1}$. To obtain an endomorphism $\mcL:\A_{n+1}^k\to\A_{n+1}^k$, we set $\mcL(x)_i := \pi(L(x)_i)$, $i=1, \ldots, k$. In this case, we write $\mcL = \pi\circ L$. For example, if $n:=3$ (the case of real quaternions) $L: \A_{4}^k\to \A_{4}^k$ and thus $\mcL = L$.

A function $f:\A_{n+1}\to\A_{n+1}$ is called a hypercomplex or paravector-valued function. Any such function is of the form $f(x_0 + \bx) = f_0(x_0, |\bx|) + \omega (\bx)f_1(x_0,|\bx|)$, where $f_0, f_1:\R\times\R^n\to\R$ and $\omega (\bx) := \frac{\bx}{|\bx|}\in \mbS^n$ with $\mbS^n$ denoting the unit sphere in $\R^n$. For some properties of  hypercomplex functions, see, for instance\cite{sproessig} .
\subsection{Iterated Function Systems and Their Attractors}
Let $(X,d)$ be a complete metrizable space with metric $d$. For a map $f: X \to X$, we define the Lipschitz constant associated with $f$ by
\[
\Lip (f) = \sup_{x,y \in X, x \neq y} \frac{d\big(f(x),f(y)\big)}{d(x,y)}.
\]
A map $f$ is said to be \emph{Lipschitz} if $\Lip (f) < + \infty$ and a \emph{contraction} if $\Lip (f) < 1$.

\begin{definition}
Let $(X,d)$ be a complete metrizable space and let $\cF$ be a finite set of contractions on $X$. Then the pair $(X,\cF)$ is called an iterated function system (IFS) on $X$.
\end{definition}
With the finite set of contractions $\cF$ on $X$, one associates a set-valued operator, again denoted by $\cF$, acting on the hyperspace ${\sH}(X)$ of nonempty compact subsets of $X$ endowed with the Hausdorff-Pompeiu metric $d_{\sH}$ by
\[
\cF (E) := \bigcup_{f\in \cF} f (E),\qquad E\in \sH(X).
\]
The Hausdorff-Pompeiu metric $d_{\sH}$ is defined by
\[
d_{\sH} (S_1, S_2) := \max\{d(S_1, S_2), d(S_2, S_1)\},
\]
where $d(S_1,S_2) := \sup\limits_{x\in S_1} d(x, S_2) := \sup\limits_{x\in S_1}\inf\limits_{y\in S_2} d(x,y)$.

The completeness of $(X,d)$ implies that the set-valued operator $\cF$ is contractive on the complete metrizable space $({\sH}(X), d_{\sH})$ with Lipschitz constant $\Lip\cF = \max\{\Lip (f) : f\in \cF\}$. Therefore, by the Banach Fixed Point Theorem, $\cF$ has a unique fixed point in $\sH(X)$. This fixed point if called the \emph{attractor} of or the \emph{fractal (set)} generated by the IFS $(X,\cF)$. The attractor or fractal $F$ satisfies the self-referential equation
\be\label{fixedpoint}
F = \cF (F) = \bigcup_{f\in\cF} f (F),
\ee
i.e., $F$ is made up of a finite number of images of itself. Eqn. \eqref{fixedpoint} reflects the fractal nature of $F$ showing that it is as an object of immense geometric complexity.

The proof of the Banach Fixed Point Theorem also shows that the fractal $F$ can be iteratively
obtained via the following procedure: Choose an arbitrary $F_0\in {\sH}(X)$ and set
\be\label{F}
F_n := \cF (F_{n-1}),\qquad n\in\N.
\ee
Then $F =\lim\limits_{n\to\infty} F_n$, where the limit is taken in the Hausdorff-Pompeiu metric $d_{\sH}$.

In Figure \ref{fig1}, two examples of fractal sets are displayed.
\begin{figure}
\begin{center}
\includegraphics[width= 3cm, height = 2.5cm]{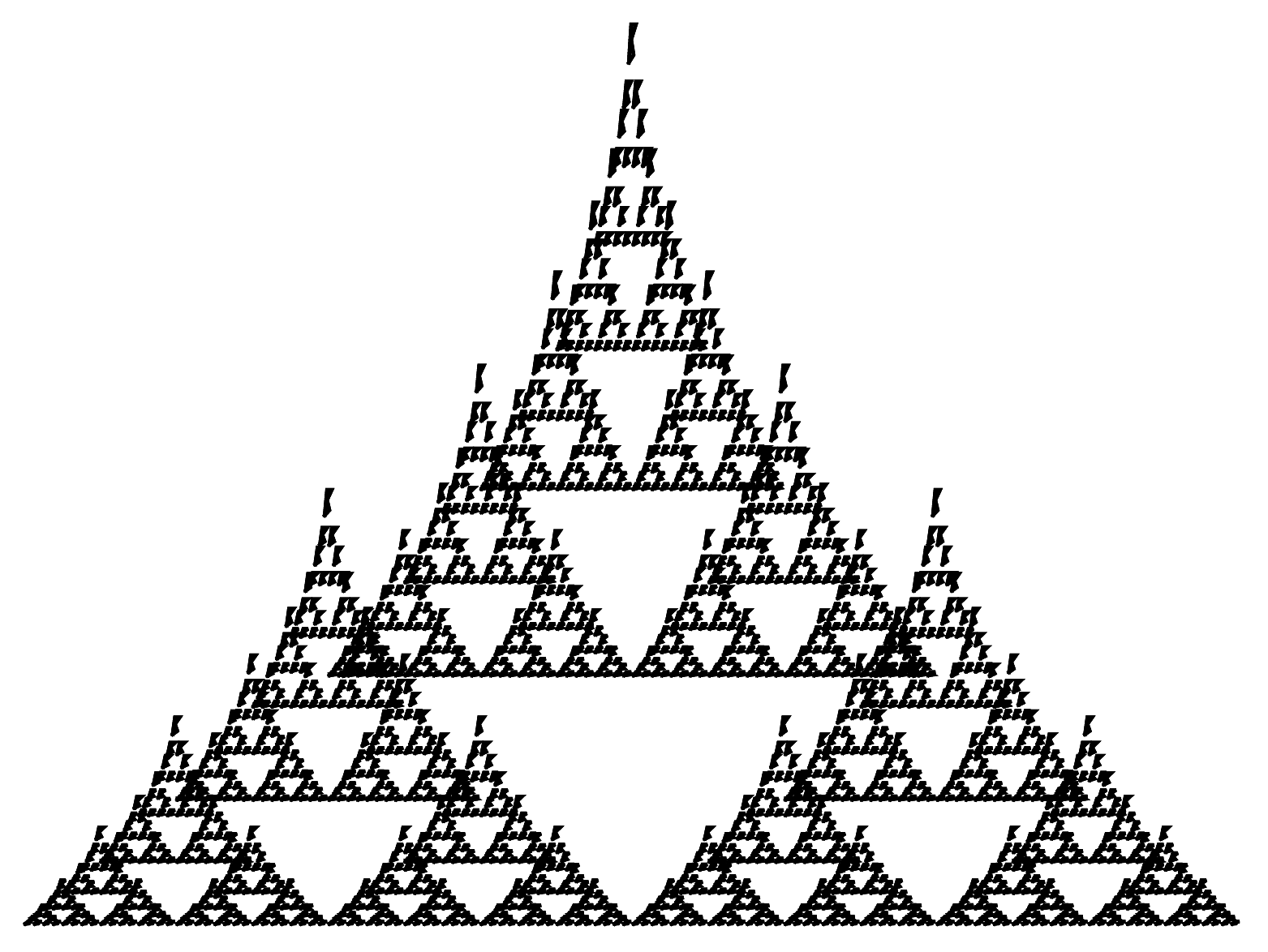}\hspace{2cm}
\includegraphics[width= 3cm, height = 2.25cm]{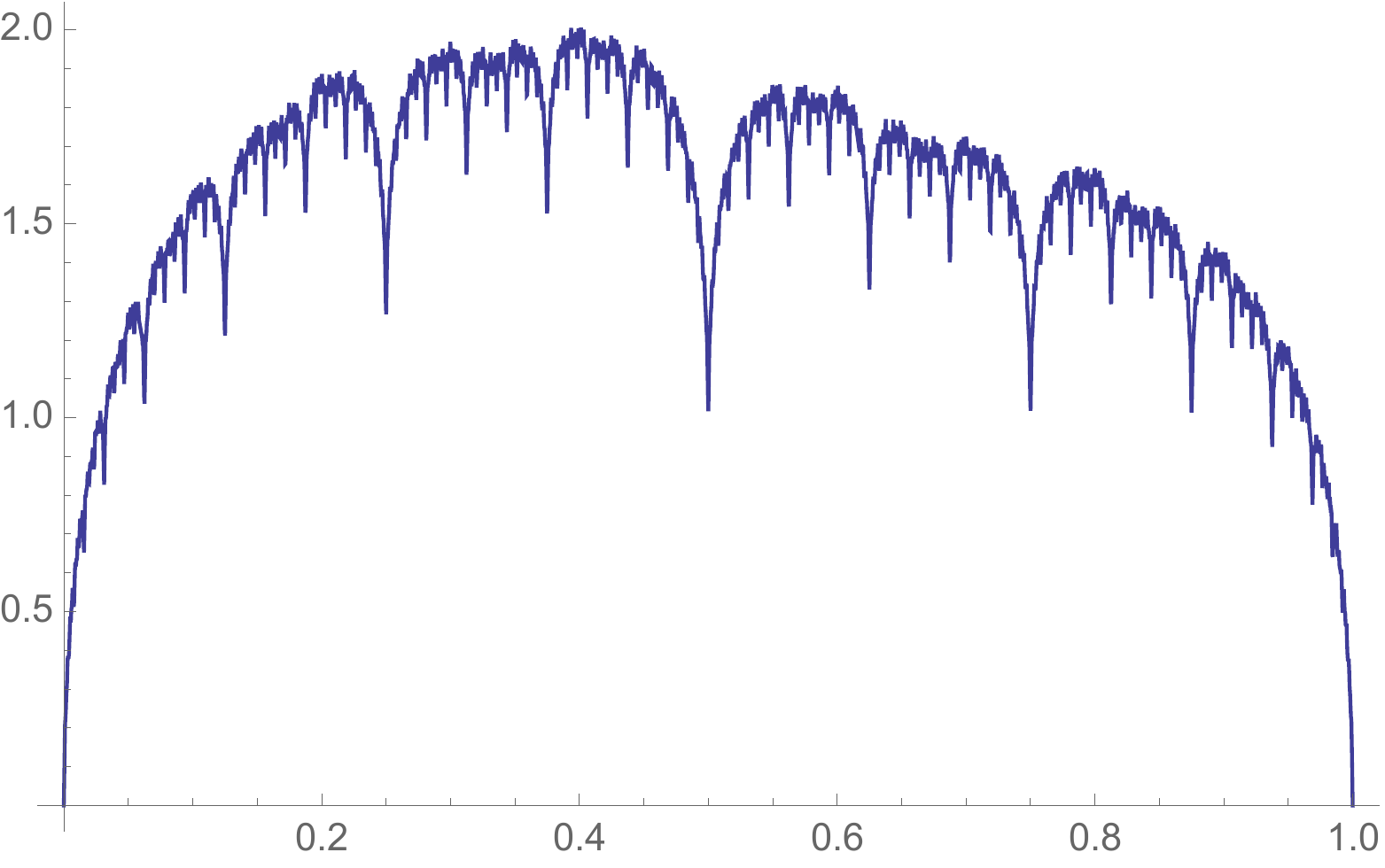}
\caption{Left: A fractal set in $X:= [0,1]^2$ generated by the maps $f_1 (x,y) := (\frac12 x, \frac12 y)$, $f_2 (x,y) := (\frac12 (x+1), \frac12 y)$, and $f_3 (x,y) := (\frac12 (x+\frac12), \frac34 y + \frac{\sqrt{3}}{4})$. Right: The graph of a \emph{fractal function} in $X:= [0,1]\times[0,3]$ generated by the maps $f_1 (x,y) := (\frac12 x, x + \frac34 y)$ and $f_2 (x,y) := (\frac12 (x+1), x^2 + \frac34 y)$.}\label{fig1}
\end{center}
\end{figure}
For more details about IFSs and fractals and their properties, we refer the interested reader to the large literature on these topics and list only two references \cite{barnsley,massopust} pertaining to the present exhibition.
\section{Hypercomplex IFSs}
Let $k\in \N$ and consider the set $\A_{n+1}^k :=  \underset{i = 1}{\overset{k}{\times}} \A_{n+1}$. We represent elements $\xi\in\A_{n+1}^k$ as column vectors. The hypercomplex conjugate ${}^*$ of $\xi\in\A_{n+1}^k$ is defined by
\[
\begin{pmatrix} \xi_1\\ \vdots\\ \xi_k
\end{pmatrix}^* := \begin{pmatrix} \overline{\xi_1} & \cdots & \overline{\xi_k}
\end{pmatrix}.
\]
Similarly, for any matrix $H = (H_{ij})$ over $\A_{n+1}$, we define $(H_{ij})^* := (\overline{H_{ji}})$. The norm of $\xi\in\A_{n+1}^k$ is defined to be
\be\label{dist}
\|\xi\| := \sqrt{\xi^*\xi} = \sqrt{\sum\limits_{i=1}^k |\xi_i|^2}.
\ee
Then $(\A_{n+1}^k,d)$ is a complete metrizable space with metric $d(\xi,\eta) := \|\xi - \eta\|$. 

We define a norm on $M_k(\A_{n+1})$ as follows. For $H\in M_k(\A_{n+1})$, set
\[
\vertiii{H} := \sup\left\{\frac{\|H\xi\|}{\|\xi\|} : 0\neq \xi\in \A_{n+1}^k\right\}.
\]
\begin{definition}
Let $k\in \N$ and let $\cF := \{f_1, \ldots, f_m\}$ be a finite collection of contractive hypercomplex functions on $\A_{n+1}^k$. Then the pair $(\A_{n+1}^k, \cF)$ is called a \emph{hypercomplex IFS} (on $\A_{n+1}^k)$. 
\end{definition}
\begin{definition}
Let $k\in \N$. An element $F\in\sH(\A_{n+1}^k)$ is called a \emph{hypercomplex attractor} of or the \emph{hypercomplex fractal (set)} generated by the hypercomplex IFS $(\A_{n+1}^k, \cF)$ if $F$ satisfies the self-referential equation
\be
F = \cF(F) = \bigcup\limits_{i=1}^m f_i (F).
\ee
\end{definition}
Note that by the Banach Fixed Point Theorem $F$ as defined above is unique.
\begin{remark}
Although the \emph{point sets} $\A_{n+1}^k$ and $(\R^{n+1})^k$ are isomorphic under an obvious bijection, the difference in the \emph{non-commutative algebraic structure} yields a broader class of attractors. (See the examples below.)
\end{remark}
As an example of a hypercomplex IFS, we consider right linear maps $L_i:\A_{n+1}^k\to C\ell(n)^k$ and define right affine maps by
\[
A_i (\xi) := L_i(\xi) + b_i = H_i \xi + b_i,
\]
with $H_i\in M_k(\A_{n+1})$ and $b_i\in \A_{n+1}^k$, $i\in\N_m$. The right affine maps $A_i$ generate hypercomplex functions $f_i$ via
\be\label{affine}
f_i (\xi):= \pi(L_i(\xi))+b_i.
\ee
It follows immediately from the definition of contraction applied to $\A_{n+1}^k$, that each $f_i$ is contractive provided that $h := \max\{\vertiii{H_i} : i\in\N_m\} < 1$. The unique fractal generated by the hypercomplex IFS $(\A_{n+1}^k, \cF)$ is then given by the nonempty compact subset $F\subset\A_{n+1}^k$ satisfying
\[
F = \bigcup_{i\in \N_m}\pi(L_i (F))+b_i.
\]
\begin{example}
Let $n := 3$ and $k:=2$. Then $\A_{4}$ can be identified with the noncommutative associative division algebra $\bH$ of quaternions with $e_0 =1$, $e_i e_j + e_j e_i = -2\delta_{ij}$, and $e_ie_j=\varepsilon_{ijk}e_k$, $i,j,k = 1,2,3$. Here, $\varepsilon_{ijk}$ denotes the Levi-Civit{\`a} symbol. We note that $\bH\cdot\bH = \bH$ and therefore $\pi(L_i) = L_i$.

On $\bH^2$ we consider the three affine mappings
\[
f_i (\xi,\eta) := \begin{pmatrix} q & 0\\ 0 & q
\end{pmatrix}\begin{pmatrix}\xi \\ \eta
\end{pmatrix} + b_i,
\]
with $b_1 := 0$, $b_2 := \begin{pmatrix}1- q \\ 0
\end{pmatrix}$, $b_3 := \begin{pmatrix}\frac12(1- q) \\ 1-q
\end{pmatrix}$, and $q := 0.3 e_0 - 0.1 e_1 + 0.4 e_2 - 0.2 e_3$. Note that $|q| = \sqrt{0.3} < 1$. Hence, the quaternionic IFS $(\bH^2, \cF)$ with $\cF=\{f_1,f_2,f_3\}$, possesses a unique attractor $F\in \sH(\bH^2)$. 

In Figure \ref{fig2}, some projections of the attractor $F$ onto subspaces of $\bH^2$ are displayed.
\begin{figure}
\begin{center}
\includegraphics[width= 3.25cm, height = 2.5cm]{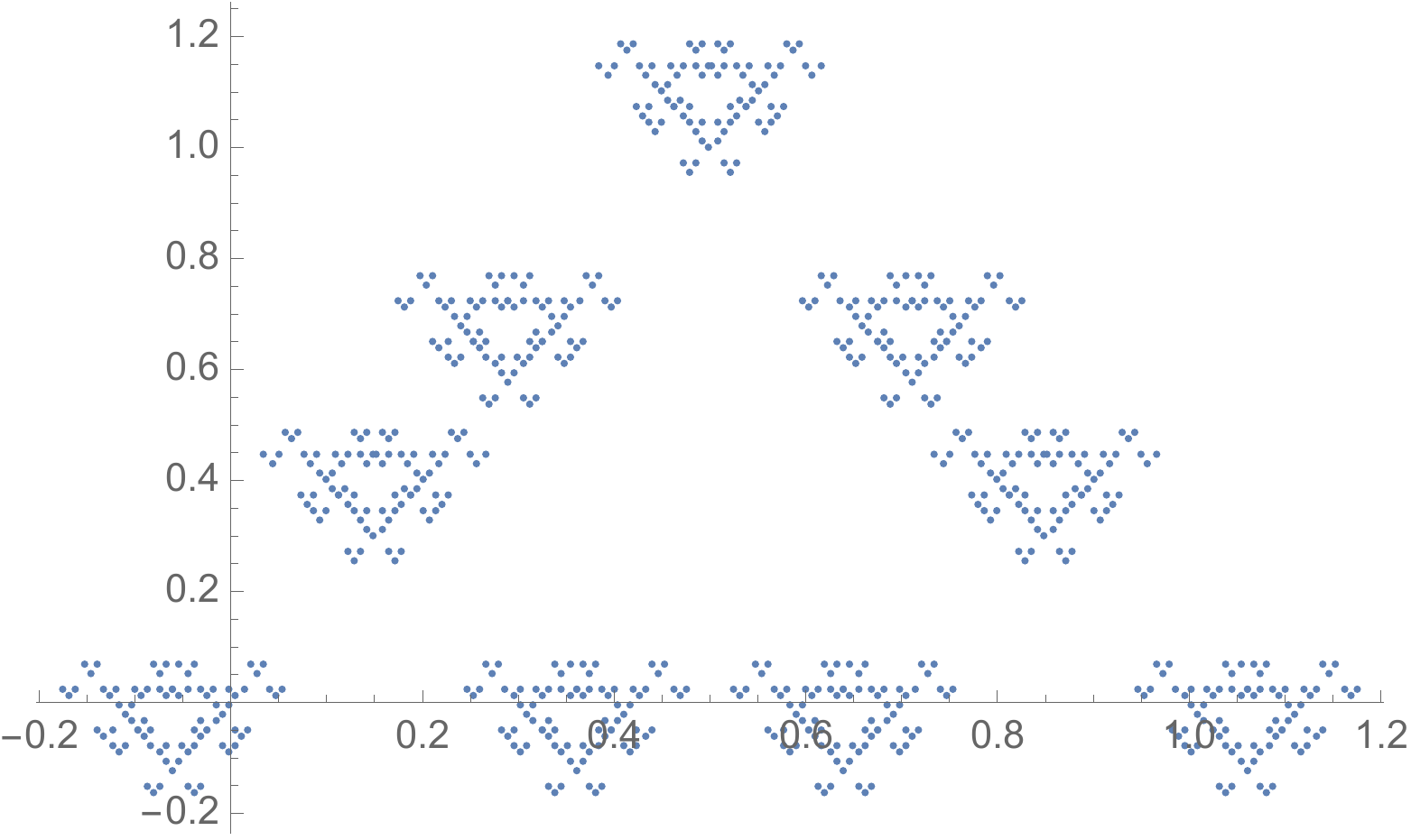}\hspace{0.75cm}
\includegraphics[width= 3.25cm, height = 2.5cm]{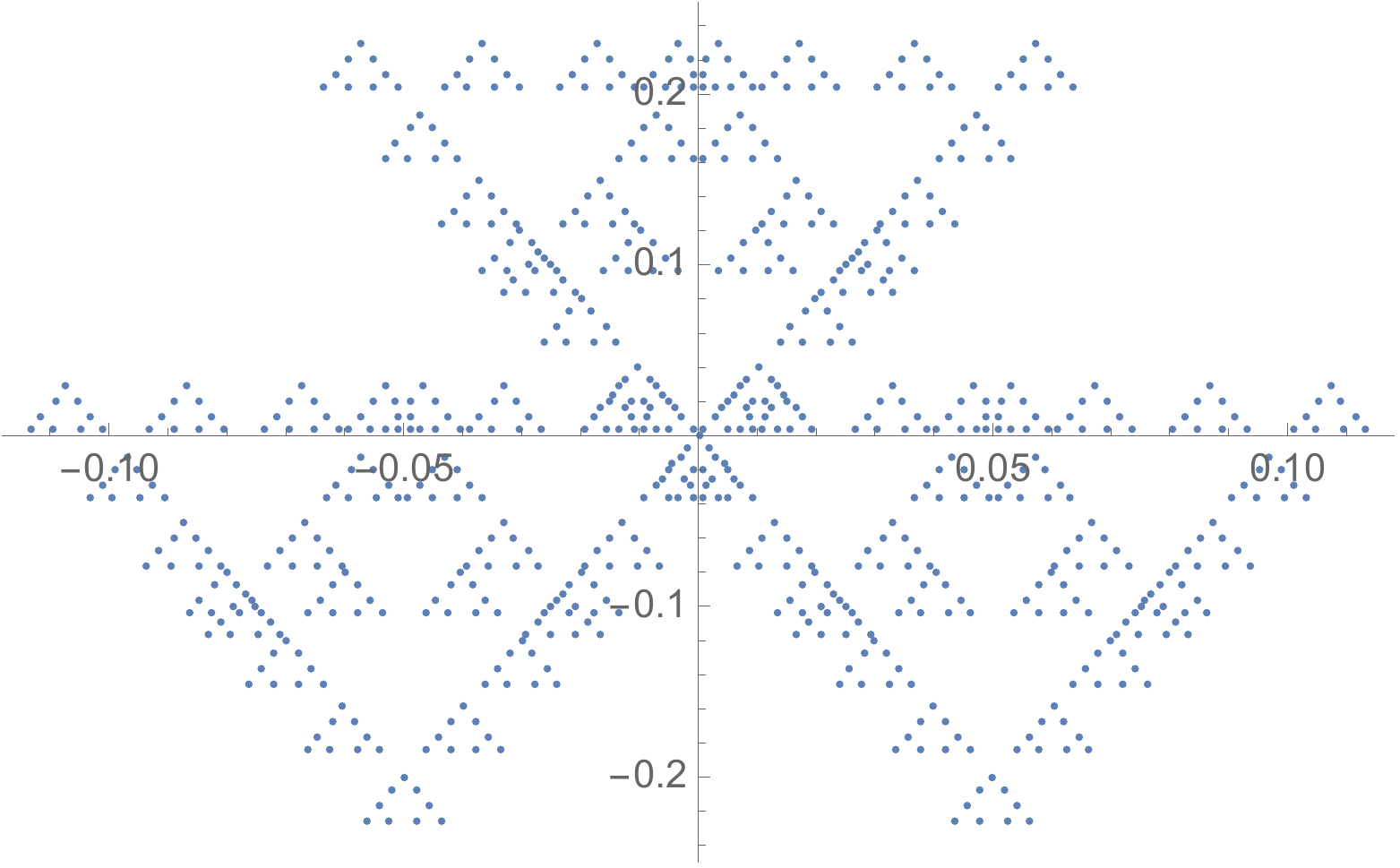}\hspace{0.75cm}
\includegraphics[width= 3.5cm, height = 2.5cm]{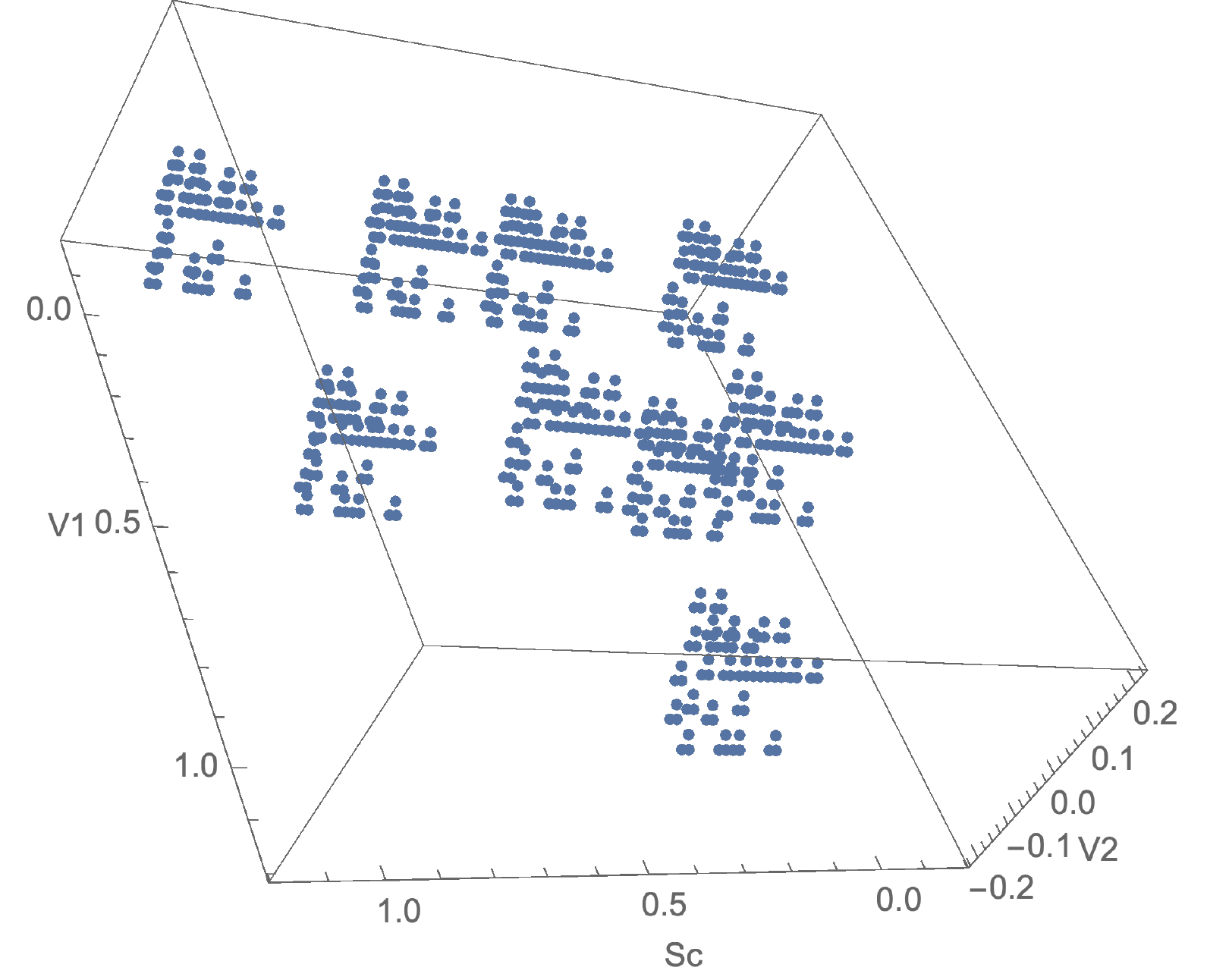}
\caption{Left: The projection of $F$ onto the $e_0-e_1$-plane. Middle: The projection of $F$ onto the $e_1-e_3$-plane. Right: The projection of $F$ onto the hyperplane spanned by $\{e_0,e_1,e_2\}$.}\label{fig2}
\end{center}
\end{figure}
\end{example}
\section{Systems of Hypercomplex Function Systems}
Employing the setting first introduced in \cite{LDV} , then generalized in \cite{DLM} , and finally applied to non-stationary IFSs in \cite{M} , we will replace the single set-valued map $\cF$ in a hypercomplex IFS by a sequence of function systems consisting of different families $\cF $ in order to define an iterative process $\{F_n\}_{n\in \N}$ with initial $F_0 \in \mathscr{H}(\A_{n+1}^k)$ as in \eqref{F}. 

To this end, consider the complete metrizable space $(\A_{n+1}^k,d)$ and let $\{\mcT_\ell\}_{\ell\in \N}$ be a sequence of transformations on $\A_{n+1}^k$, i.e., $\mcT_\ell:  \A_{n+1}^k\to \A_{n+1}^k$.

\begin{definition}
A subset $\mathscr{I}$ of $\A_{n+1}^k$ is called a hypercomplex invariant set of the sequence $\{\mcT_\ell\}_{\ell\in \N}$ if
\[
\forall\,\ell\in \N\;\forall\,x\in \mathscr{I}: \mcT_\ell (x)\in \mathscr{I}.
\]
\end{definition}
A criterion for obtaining a hypercomplex invariant domain for a sequence $\{\mcT_\ell\}_{\ell\in \N}$ of transformations on a complete metrizable space is the following which first appeared in \cite{LDV} . We will state the result for our setting.

\begin{proposition}
Let $\{\mcT_\ell\}_{\ell\in \N}$ be a sequence of transformations on the complete metrizable space $(\A_{n+1}^k,d)$. Assume that there exists an $x_0\in \A_{n+1}^k$ such that for all $x\in \A_{n+1}^k$
\be\label{cond}
d(\mcT_\ell(x),x_0) \leq s\,d(x,x_0) + M,
\ee
for some $s\in [0,1)$ and $M> 0$. Then any ball $B_r(x_0) := \{x\in  \A_{n+1}^k : |x - x_0| < r\}$ with radius $r > M/(1-s)$ is a hypercomplex invariant set for $\{\mcT_\ell\}_{\ell\in \N}$.
\end{proposition}
\begin{proof}
The proof of this statement follows directly from \cite[Lemma 3.7]{LDV} and \cite[Remark 3.8]{LDV}.
\end{proof}

In case the transformations $\mcT_\ell$ are maps of the form \eqref{affine} then condition \eqref{cond} is satisfied with $x_0 = 0$, $M:= \sup\limits_{\ell\in \N}\|b_\ell\| < \infty$, and any $H_\ell$ with $\sup\limits_{\ell\in \N} \vertiii{H_\ell}=: s < 1$. For, if we choose $x_0:= 0$,
\begin{align*}
d(\mcT_\ell (x), 0) &= \|\mcT_\ell (x)\| = \|H_\ell x + b_\ell\| \leq \|H_\ell x\| + \|b_\ell\|\\
& \leq \vertiii{H_\ell}\cdot\|x\| + \|b_\ell\|.
\end{align*}
Hence, every ball centered at the origin of $\A_{n+1}^k$ of radius greater than $M/(1-s)$ is a hypercomplex invariant set for $\{\mcT_\ell\}_{\ell\in \N}$.

Now suppose that $\{\cF_\ell\}_{\ell\in \N}$ is a sequence of set-valued maps $\cF_\ell:\sH(\A_{n+1}^k)\to \sH(\A_{n+1}^k)$ defined by
\be\label{2.1}
\cF_\ell (F_0) := \bigcup_{i=1}^{n_\ell} f_{i,\ell} (F_0), \quad F_0\in \sH(\A_{n+1}^k),
\ee
where $\cF_\ell = \{f_{i,\ell}: i\in \N_{n_\ell}\}$ is a family of hypercomplex contractions constituting a hypercomplex IFS on the complete metrizable space $(\A_{n+1}^k,d)$. Setting $s_{i,\ell} := \Lip (f_{i,\ell})$, we obtain that $\Lip (\cF_\ell) = \max\{s_{i,\ell}: i\in \N_{n_\ell}\} < 1$.

The following definition is taken from \cite[Section 4]{LDV} using our setting with $X = \A_{n+1}^k$.

\begin{definition}
Let $F_0\in \sH(\A_{n+1}^k)$. For any $\ell\in \N$, the sequence
\be
\Psi_\ell (F_0) := \cF_1\circ\cF_{2} \circ \cdots \circ \cF_\ell (F_0)
\ee
is called the backward trajectory of $F_0$.
\end{definition}

\noindent
Backward trajectories converge under rather mild conditions 
and their limits generate new types of fractal sets. For more details, we refer the interested reader to \cite{LDV,DLM,M} .

The next theorem summarizes the convergence result for backward trajectories. 

\begin{theorem}
Let $\{\cF_\ell\}_{\ell\in\N}$ be a family of set-valued maps of the form \eqref{2.1} whose elements are collections $\cF_\ell = \{f_{i,\ell}: i\in \N_{n_\ell}\}$ of hypercomplex contractions constituting hypercomplex IFSs on the complete metrizable space $(\A_{n+1}^k,d)$. Suppose that 
\begin{enumerate}
\item[\emph{(i)}] there exists a nonempty closed hypercomplex invariant set $\mathscr{I}\subseteq \A_{n+1}^k$ for $\{f_{i,\ell}\}$, $i\in \N_{n_\ell}$, $\ell\in \N$,
\item[\emph{(ii)}] and
\be\label{lipcond}
\sum_{\ell=1}^\infty\prod_{j=1}^\ell \Lip (\cF_j) < \infty.
\ee
\end{enumerate}
Then the backward trajectories $\{\Psi_\ell (F_0)\}$ converge for any initial $F_0\subseteq \mathscr{I}$ to a unique hypercomplex attractor $F\subseteq \mathscr{I}$.
\end{theorem}
\begin{proof}
The proof can be found in \cite{LDV} . Note that, as the proof only involves point sets and moduli of numbers, it immediately applies to the hypercomplex setting.
\end{proof}
\begin{example}
Here, we take $n := 3$ and $k:=1$ and write $\bH$ for $\A_{4}^1$. Define a family $\{\cF_\ell\}_{\ell\in\N}$ of set-valued epimorphism on $\sH$ whose members are as follows:
\begin{align*}
\cF_\ell := \begin{cases}
\cF_1, & 10(j-1) < \ell \leq 10j -5,\\
\cF_2, & 10j-5 < \ell \leq 10j,
\end{cases}\quad j\in \N,
\end{align*}
where
\begin{align*}
\cF_1 &:= \{q_1 x, q_1 x + \mathbbm{1}-q_1\},\\
\cF_2 & := \{0.7 \hat{q_2} x \hat{q_2} + 0.1, 0.7 \hat{q_2} x \hat{q_2} - 0.1\},
\end{align*}
with $q_1:= 0.75e_0$, $\mathbbm{1}:= e_0+e_1+e_2+e_3$, and $\hat{q_2} := \sqrt{\frac{10}{3}}(0.3e_0-0.1e_1+0.4e_2-0.2e_3)$.

Figure \ref{fig3} displays the projection of the attractor $F$ of the backward trajectory $\{\Psi_\ell\}_{\ell\in \N}$ onto the $e_0-e_1$-plane at two different levels. 
\begin{figure}
\begin{center}
\includegraphics[width= 3cm, height = 2cm]{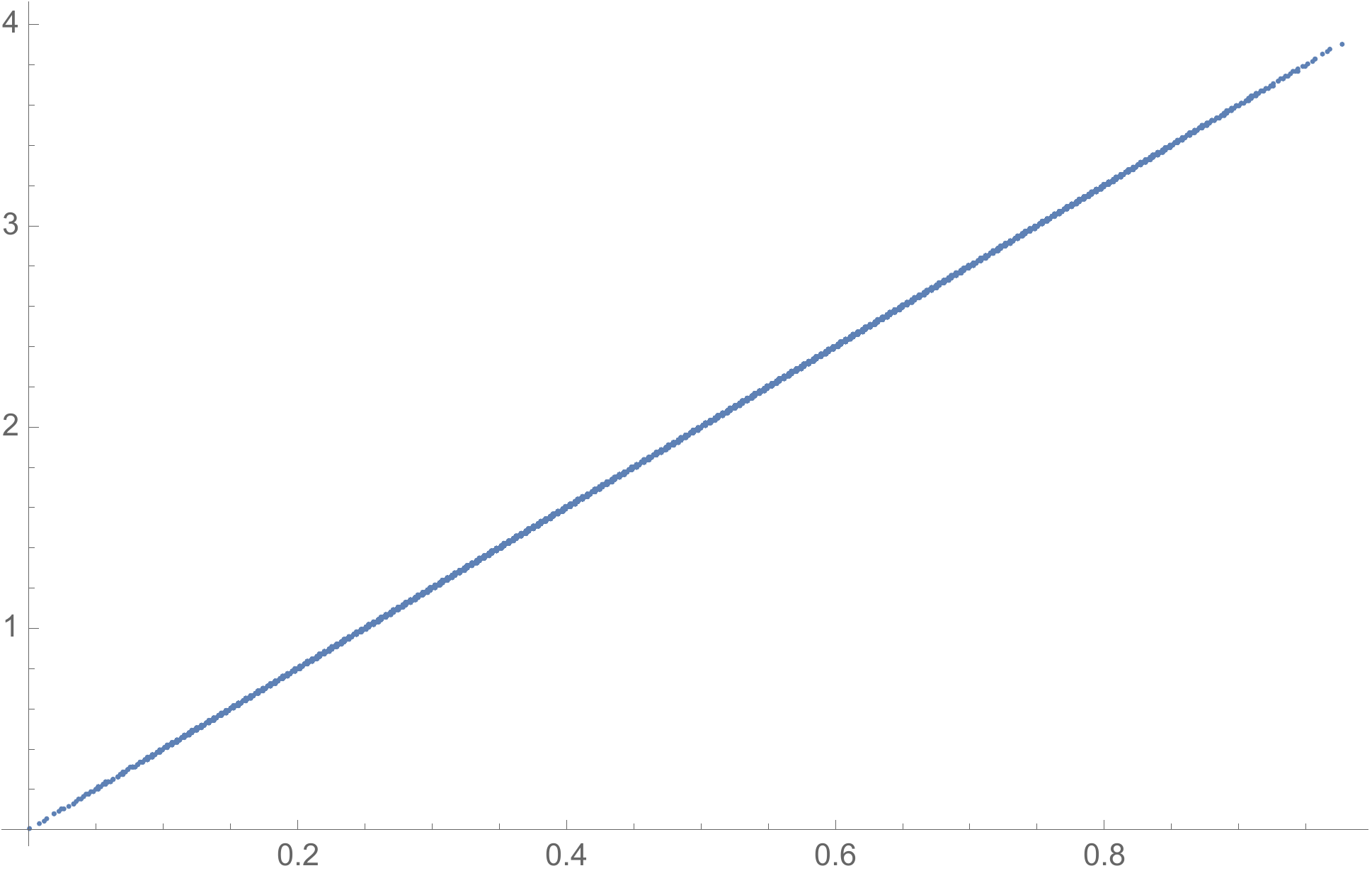}\hspace{1.5cm}
\includegraphics[width= 3.25cm, height = 2.5cm]{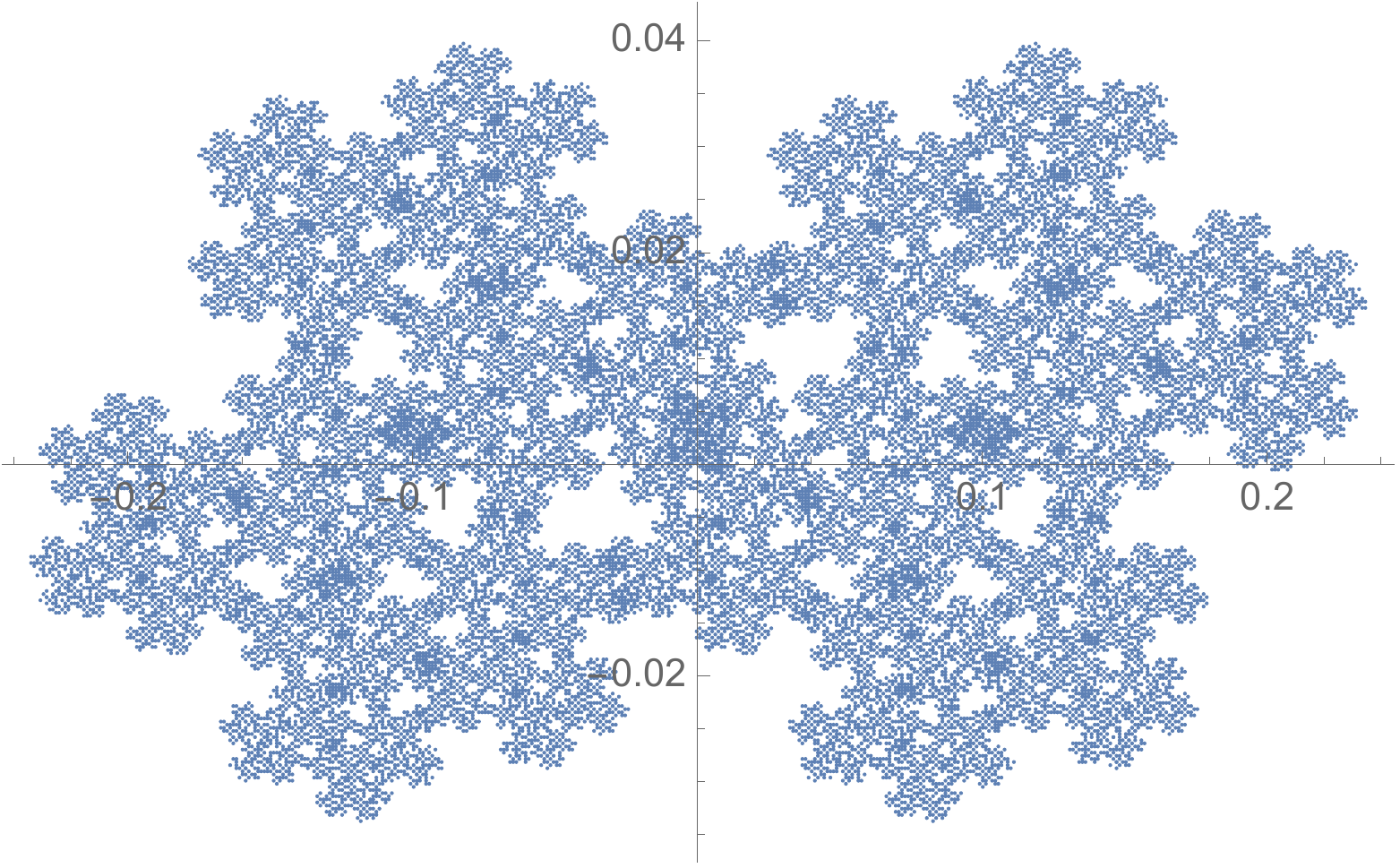}\caption{The projection of the attractor $F$ of the backward trajectory $\{\Psi\}_{k\in \N}$ onto the $e_0-e_1$-plane at two different levels. The attractor is smooth at one level (left) and fractal at another (right).} \label{fig3}
\end{center}
\end{figure}

This example shows that hypercomplex attractors generated by backwards trajectories exhibit more flexibility in their shapes and that a proper choice of IFSs reveals different local behavior. This is due to the fact that in the sequence
\[
\cF_1\circ\cF_2\circ \cdots \cF_{\ell-1}\circ\cF_\ell (F_0), \quad F_0\in \sH(\A_{n+1}^k),
\]
the global shape of the hypercomplex attractor is determined by the initial maps $\cF_1\circ \cF_2\circ \cdots$ whereas the local shape is determined by the terminal maps $\cF_{\ell-1}\circ\cF_\ell\circ\cdots$. In addition, the non-commutative algebraic structure of $\bH$ yields a broader class of attractors than in the case $\R^4$.
\end{example}

%%%%%%%%%%%%%%%%%%%%%%%% referenc.tex %%%%%%%%%%%%%%%%%%%%%%%%%%%%%%
% sample references
% %
% Use this file as a template for your own input.
%
%%%%%%%%%%%%%%%%%%%%%%%% Springer-Verlag %%%%%%%%%%%%%%%%%%%%%%%%%%
%
% BibTeX users please use
% \bibliographystyle{}
% \bibliography{}
%

\end{document}